\documentclass[11pt]{amsart}
\usepackage{amssymb}
\usepackage{amsmath}
\newtheorem{theorem}{Theorem}[section]
\newtheorem{lemma}[theorem]{Lemma}
\newtheorem{remark}[theorem]{Remark}
\newtheorem{definition}[theorem]{Definition}

\newtheorem{corollary}[theorem]{Corollary}
\newtheorem{proposition}[theorem]{Proposition}

\newtheorem{lem-def}[theorem]{Lemma-Definition}

\def\as#1{\renewcommand\arraystretch{#1}}
\def\ax{A[x]}
\def\a1x{A_1[x]}
\def\dsc{\operatorname{disc}}
\def\F{\mathbb F}
\def\fph{\F_{\phi}}
\def\gen#1{\big\langle\, {#1} \,\big\rangle}

\def\ind{\operatorname{ind}}
\def\iso{\ \lower .6ex\hbox{$\stackrel{\lra}{\mbox{\tiny $\sim\,$}}$}\ }

\def\lra{\longrightarrow}
\def\m{\mathfrak{m}}
\def\md#1{\ \mbox{\rm(mod }{#1})}
\def\mn{\operatorname{Min}}

\def\nph{\operatorname{N}_{\phi}}
\def\npp{\operatorname{N}_{\phi}^-}
\def\op{\operatorname}
\def\ord{\operatorname{ord}}
\def\p{\mathfrak{p}}
\def\pp{\mathcal{P}}

\def\Q{\mathbb Q}
\def\qb{\overline{\Q}}
\def\R{\mathbb R}
\def\rd{\operatorname{red}}
\def\t{\theta}

\def\v2{v_2^{(2)}}
\def\Z{\mathbb Z}
\def\zp{\Z_{(p)}}
\def\zpx{\Z_p[x]}
\newtheorem{alg}[theorem]{Algorithm}
\newlength{\alginputwidth}
\newlength{\algtmp}
\newcommand{\Algo}[5]
            {
            \begin{alg}[#1] \label{#2}{$\;$}\rm
                \\
\mbox{\enspace}
                \rlap{\rm Input: }\phantom{\rm Output: }
\parbox[t]{\alginputwidth}{#3}
                \\
\mbox{\enspace}
                {\rm Output: }
\parbox[t]{\alginputwidth}{#4}
\parskip0pt
\begin{list}{}{\setlength{\leftmargin}{0pt}}
\item                #5
\end{list}
            \end{alg}
            \goodbreak}
\title[Local-to-global computation of integral bases]{Local-to-global computation of integral bases without a previous factorization of the discriminant}

\author{Jordi Gu\`ardia}
\address{Departament de Matem\`atica Aplicada IV, \ Escola Polit\`ecnica Superior d'Enginyeria de Vilanova i la Geltr\'u, Av. V\'\i ctor Balaguer s/n. E-08800 Vilanova i la Geltr\'u, Catalonia, Spain}
\email{guardia@ma4.upc.edu}

\author{Enric Nart}

\address{Departament de Matem\`{a}tiques,
         Universitat Aut\`{o}noma de Barcelona,
         Edifici C\\ E-08193 Bellaterra, Barcelona, Catalonia, Spain}
\email{nart@mat.uab.cat}

\keywords{discriminant; integral basis; Newton polygon; residual polynomial}

\thanks{Partially supported by MTM2012-34611 and MTM2013-40680 from the Spanish MEC}

\makeatletter
\@namedef{subjclassname@2010}{%
\textup{2010} Mathematics Subject Classification}

\subjclass[2010]{Primary 11R04; Secondary 11Y40}

\begin{document}
\begin{abstract}
We adapt an old local-to-global technique of Ore to compute, under certain mild assumptions, an integral basis of a number field without a previous factorization of the discriminant of the defining polynomial. In a first phase, the method yields as a by-product successive splittings of the discriminant. When this phase concludes, it requires a squarefree factorization of some base factors of the discriminant to terminate. 
\end{abstract}

\maketitle

%\ccode{2010 Mathematics Subject Classification: 11R04, 11S15}

\section*{Introduction}
In his 1923 PhD thesis and a series of subsequent papers, Ore used Newton polygon techniques to solve some basic arithmetic tasks in number fields, such as prime ideal decomposition or the construction of local integral bases  \cite{ore23,ore24,ore25,ore26,ore28}. 
From a computational perspective, Ore's methods are very efficient but they work only under certain mild assumptions. 

Let $f\in\Z[x]$ be a monic irreducible polynomial of degree $n>1$. Let $K=\Q[x]/(f)$ be the corresponding number field and denote by $\Z_K$ its ring of integers. 
In order to find the prime ideal decomposition of a prime number $p$, or to construct a $p$-integral basis of $\Z_K$, the defining polynomial $f$ must be  \emph{$p$-regular} (Definition \ref{pregular}).

Ore's method performs three \emph {classical dissections}, yielding a successive factorization of $f$ over $\Z_p[x]$. The $p$-regularity condition ensures that all $p$-adic factors of $f$ that have been found after these dissections are irreducible. 

These techniques are of a local nature, but the computation of a global integral basis of $K$ (a basis of $\Z_K$ as a $\Z$-module) may be derived from the local $p$-bases for $p$ running on the prime factors of the discriminant $\dsc(f)$ of $f$.  However, this local-to-global approach requires the factorization of $\dsc(f)$, which is impossible (or an extremely heavy task) if the degree of $f$ is large and/or it has large coefficients. 
 
In this paper, we show that the three classical dissections may be applied to find an $N$-integral basis, for any given positive divisor $N$ of the discriminant, under a similar assumption of $N$-regularity. An $N$-integral basis is a family $\alpha_1,\dots,\alpha_n\in\Z_K$, which is simultaneously a $p$-integral basis for all prime divisors $p$ of $N$. 

The $N$-integral basis is derived immediately from some data collected along the execution of the three classical dissections, in a complete analogy with the \emph{method of the quotients} introduced in \cite{np4}.

We follow Lenstra's strategy as in the elliptic curve factorization algorithm. We proceed as if $N$ were a prime; if this causes no trouble we get a candidate for an $N$-integral basis, but if at some step the method crashes then it yields a proper divisor of $N$. In the latter case, we may write $N=N_1^{e_1}\cdots N_k^{e_k}$ with some coprime base factors $N_1,\dots N_k$, and we may start over to construct $N_i$-integral bases for all $i$.  Once this phase concludes, we get $m$-integral bases of all base factors $m$ of $N$ obtained so far, if these base factors $m$ are squarefree. If this is not the case, we must find their squarefree factorization to continue the procedure. 

Once we get $m$-integral bases of all coprime base factors $m$ of $N$, they may be patched together, in a well-known way, to provide an $N$-integral basis. 
When applied to $N=\dsc(f)$, this procedure computes a global integral basis of $K$.

With respect to the traditional local-to-global approach, this method has a double advantage. On one hand, it does not require a previous factorization of the discriminant; on the other hand, it requires squarefree factorization of some integers $m$, but after several splittings of the discriminant, these base factors $m$ may be much smaller than the discriminant itself.

It is quite plausible that the assumption of $N$-regularity may be dropped and the methods of this paper lead to an unconditional local-to-global computation of integral bases. In fact, inspired in some work by MacLane \cite{mcla,mclb}, Montes designed an algorithm to perform successive dissections beyond the three classical ones, to obtain an unconditional $p$-adic factorization of $f$ \cite{GMN,m}. Also, in the paper \cite{bases}, the method of the quotients was adapted to this general situation to provide a concrete procedure to construct $p$-integral bases from data collected along the execution of all these dissections. Thus, the only work to be done is the extension of the ideas of this paper to Montes' ``higher order" dissections.

\section*{Acknowledgements}
It was Claus Fieker who suggested to us that an adequate development of Ore's methods modulo $N$ could lead to these kind of results. We are indebted to him for his fine intuition.    

\section{The three classical dissections}\label{secOre}
In this section we recall some results of Ore on arithmetic applications of Newton polygons \cite{ore23,ore24,ore25,ore26,ore28}. Modern proofs of these results can be found in \cite[Sec. 1]{GMN}.

We fix from now on a monic irreducible polynomial $f\in\Z[x]$ of degree $n>1$, and a prime number $p$.

Take a root $\t\in\qb$ of $f$. Let $K=\Q(\t)$ be the number field generated by $\t$ and denote by $\Z_K$ the ring of integers of $K$. 

Let $\Z_p$ be the ring of $p$-adic integers, $\Q_p$ the fraction field of $\Z_p$ and $\qb_p$ an algebraic closure of $\Q_p$. 
We denote by $v_p\colon \qb_p^{\,*}\longrightarrow \Q$, the $p$-adic valuation normalized by $v_p(p)=1$.

Let $\pp$ be the set of prime ideals of $K$ lying above $p$.
For any $\p\in\pp$, we denote by $v_{\p}$ the discrete valuation of $K$ associated to $\p$, and by $e(\p/p)$ the ramification index of $\p$. Endow $K$ with the $\p$-adic topology and fix a topological embedding $\iota_\p\colon K\hookrightarrow \qb_p$; we then have,
$$
v_\p(\alpha)=e(\p/p)v_p(\iota_\p(\alpha)),\quad\forall \,\alpha\in K.
$$

By a celebrated theorem of Hensel, there is a canonical bijection between $\pp$ and the set of monic irreducible factors of $f$ in $\Z_p[x]$. The irreducible factor attached to a prime ideal $\p$
is the minimal polynomial $F_\p\in\Z_p[x]$ of $\iota_\p(\t)$ over $\Q_p$. 

\subsection{First dissection: Hensel's lemma}
Let us choose monic polynomials $\phi_1,\dots,\phi_t\in\Z_p[x]$ whose reduction modulo $p$ are the pairwise different irreducible factors of $f$ modulo $p$. We then have a decomposition:
$$
f\equiv \phi_1^{\ell_1}\dots\;\phi_t^{\ell_t} \md{p},
$$for certain positive exponents $\ell_1,\dots,\ell_t$. By Hensel's lemma, $f$ decomposes in $\Z_p[x]$ as:
$$
f=F_1\cdots F_t,
$$for certain monic factors $F_i\in\Z_p[x]$ such that $F_i\equiv \phi_i^{\ell_i} \md{p}$.

\begin{definition}For all $1\le i\le t$, define
$\pp_{\phi_i}:=\{\p\in\pp\mid v_p(\phi_i(\iota_\p(\t)))>0\}$.
\end{definition}

Since the polynomials $\phi_1,\dots,\phi_t$ are pairwise coprime modulo $p$, the set $\pp$ splits as the disjoint union: 
$\ \pp=\pp_{\phi_1}\coprod\cdots\coprod \pp_{\phi_t}$.

\subsection{Second dissection: Newton polygons}\label{subsec2nd}
Let us fix one of the $p$-adic factors of $f$; say, $F=F_i$ for some $1\le i \le t$. Accordingly, denote $$\phi=\phi_i, \quad\ell=\ell_i.$$ 

The aim of the second dissection is to obtain a further splitting of $F$ in $\Z_p[x]$, or equivalently, a further dissection of the set $\pp_\phi$.

Let us extend $v_p$ to a discrete valuation of $\Q_p(x)$ by letting it act in the following way on polynomials:
$$
v_p(a_0+a_1x+\cdots+a_sx^s+\cdots)=\mn_{0\le s}\{v_p(a_s)\}.
$$

Our defining polynomial $f$ admits a unique $\phi$-expansion:
$$
f=a_0+a_1\phi+\cdots+a_r\phi^r,
$$
with $a_i\in\Z_p[x]$ having $\deg a_i<\deg\phi$. For any coefficient $a_i$ we compute the $p$-adic value $u_i=v_p(a_i)\in\Z\cup\{\infty\}$.

\begin{definition}
The $\phi$-Newton polygon of $f$ is the lower convex hull of the set of all points $(i,u_i)$, $u_i<\infty$, in the Euclidian plane.
We denote this open convex polygon by $\nph(f)$.
\end{definition}

The \emph{length} of this polygon is by definition the abscissa of the last vertex. We denote it by  $\ell(\nph(f))=r=\lfloor \deg(f)/\deg\phi\rfloor$. The typical shape of this polygon is shown in Figure \ref{fig1}.

\begin{figure}
\begin{center}
\setlength{\unitlength}{5.mm}
\begin{picture}(13,7)
\put(6.85,.85){$\bullet$}\put(5.85,-.15){$\bullet$}\put(3.85,2.85){$\bullet$}
\put(2.85,1.85){$\bullet$}\put(1.85,3.85){$\bullet$}\put(2.35,4.85){$\bullet$}\put(.85,5.85){$\bullet$}
\put(0,0){\line(1,0){12}}\put(1,-1){\line(0,1){8}}\put(10.85,-.15){$\bullet$}
\put(6,0){\line(-3,2){3}}\put(3,2){\line(-1,2){2}}\put(6,.03){\line(-3,2){3}}
\put(3,2.03){\line(-1,2){2}}\put(11,0){\line(-1,0){5}}\put(11,.02){\line(-1,0){5}}
\put(10.9,-.8){\begin{footnotesize}$r$\end{footnotesize}}
\put(6,-.8){\begin{footnotesize}$\ell$\end{footnotesize}}
\put(.6,-.6){\begin{footnotesize}$0$\end{footnotesize}}
\put(6,4.6){\begin{footnotesize}$\nph(f)$\end{footnotesize}}
\end{picture}\qquad\qquad
\begin{picture}(8,8)
\put(4.85,-.15){$\bullet$}\put(2.85,2.85){$\bullet$}\put(1.35,4.85){$\bullet$}
\put(1.85,1.85){$\bullet$}\put(.85,3.85){$\bullet$}\put(-.15,5.85){$\bullet$}
\put(-1,0){\line(1,0){8}}\put(0,-1){\line(0,1){8}}
\put(5,0){\line(-3,2){3}}\put(2,2){\line(-1,2){2}}\put(5,.03){\line(-3,2){3}}
\put(2,2.03){\line(-1,2){2}}
\put(5,-.8){\begin{footnotesize}$\ell$\end{footnotesize}}
\put(-.4,-.6){\begin{footnotesize}$0$\end{footnotesize}}
\put(2.6,4.6){\begin{footnotesize}$\npp(f)$\end{footnotesize}}
\end{picture}\caption{}\label{fig1}
\end{center}
\end{figure}

The $\phi$-Newton polygon is the union of different adjacent \emph{sides}, whose endpoints are called \emph{vertices} of the polygon.

\begin{definition}
The polygon $\npp(f)$ determined by the sides of negative slope of $\nph(f)$ is called the \emph{principal $\phi$-polygon} of $f$. \end{definition}

Note that the length of $\npp(f)$ is equal to $\ell$, the order with which the reduction of $\phi$ modulo $p$ divides the reduction of $f$ modulo $p$.

\begin{theorem}\label{Diss2} 
With the above notation, suppose that the principal Newton polygon $\npp(f)$ has $k$ different sides with slopes $-\lambda_1<\cdots<-\lambda_k$. Then $F$ admits a factorization in $\zpx$ into a product of $k$ monic polynomials
$$
F=G_1\cdots G_k,
$$
such that, for all $1\le j\le k$, 
\begin{enumerate}
\item $\nph(G_j)$ is one-sided, with slope $-\lambda_j$,
\item All roots $\alpha\in\qb_p$ of $G_j$ satisfy $v_p(\phi(\alpha))=\lambda_j$. 
\end{enumerate}
\end{theorem}

In particular, the set $\pp_\phi$ splits as the disjoint union: 
$$
\pp_\phi=\pp_{\phi,\lambda_1}\coprod\cdots\coprod \pp_{\phi,\lambda_k},\qquad\pp_{\phi,\lambda}:=\left\{\p\in\pp_\phi\mid v_p(\phi(\iota_\p(\t)))=\lambda\right\}.
$$

\subsection{Third  dissection: residual polynomials}\label{subsec3rd}
Keeping with the above notation, we fix one of the $p$-adic factors of $F$; say, $G=G_j$ for some $1\le j \le k$. Accordingly, denote $\lambda=\lambda_j$. 

The aim of the third dissection is to obtain a further splitting of $G$ in $\Z_p[x]$, or equivalently, a further dissection of the set $\pp_{\phi,\lambda}$.

By construction, the points $(i,u_i)$ lie all on or above $\npp(f)$. The set of points $(i,u_i)$ that lie on $\npp(f)$ contain the arithmetic information we are interested in. 

Consider the maximal ideal $(p,\phi)$ of $\Z_p[x]$ and denote by 
$$ \rd_{p,\phi}\colon \Z_p[x]\lra \F_{p,\phi}:=\Z_p[x]/(p,\phi)$$
the homomorphism of reduction modulo $(p,\phi)$.
We attach to any integer abscissa $0\le i\le \ell$ the following \emph{residual coefficient} $c_i\in\F_{p,\phi}$:
$$\as{1.6}
c_i=\left\{\begin{array}{ll}
0,&\mbox{ if $(i,u_i)$ lies above $\npp(f)$, or }u_i=\infty,\\\rd_{p,\phi}\left(\dfrac{a_i}{p^{u_i}}\right),&\mbox{ if $(i,u_i)$ lies on }\npp(f).
\end{array}
\right.
$$ Note that $c_i$ is always nonzero in the latter case, because $\deg a_i<\deg\phi$. 

Let $S$ be the side of $\npp(f)$ with slope $-\lambda=-h/e$, where $h,e$ are positive coprime integers. We introduce the following notation:
\begin{enumerate}
\item $\ell(S)$ is the length of the projection of $S$ to the $x$-axis,
\item $d(S):=\ell(S)/e$.
\end{enumerate} 

Note that $S$ is divided into $d(S)$ segments by the
points of integer coordinates that lie on  $S$. 

\begin{definition}
Let $s$ be the abscissa of the left endpoint of $S$, and let $d=d(S)$. The \emph{residual polynomial} attached to $S$ (or to $\lambda$) is defined as:
$$
R_{\phi,\lambda}(f)=c_s+c_{s+e}\,y+\cdots+c_{s+(d-1)e}\,y^{d-1}+c_{s+de}\,y^d\in\F_{p,\phi}[y].
$$
\end{definition}

Note that $c_s$ and $c_{s+de}$ are always nonzero, so that the residual polynomial has degree $d$ and is never  divisible by $y$.

\begin{theorem}\label{Diss3} 
With the above notation, suppose that $R_{\phi,\lambda}(f)$ decomposes
$$R_{\phi,\lambda}(f)=c\,\psi_{1}^{n_{1}}\cdots \psi_{r}^{n_{r}},\quad c\in\F_{p,\phi}^*,$$into a product of powers of pairwise different monic irreducible polynomials in $\F_{p,\phi}[y]$. Then, the polynomial $G$ has a further factorization in $\Z_p[x]$ into a product of $r$ monic polynomials
$$
G=H_{1}\cdots H_{r},
$$   
such that  $\nph(H_i)$ is one-sided of slope $\lambda$ and $R_{\phi,\lambda}(H_i)=\psi_{i}^{n_{i}}$, for all $i$. 

Finally, if  $n_{i}=1$, the polynomial $H_{i}$ is irreducible in $\Z_p[x]$ and the ramification index and residual degree of the $p$-adic field $K_{i}=\Z_p[x]/(H_i)$ are given by
$\
e(K_{i}/\Q_p)=e$, $\, f(K_{i}/\Q_p)=\deg\phi\cdot\deg\psi_{i}$.\hfill{$\Box$}
\end{theorem}

In particular, the set $\pp_{\phi,\lambda}$ splits as the disjoint union: 
$$
\pp_{\phi,\lambda}=\pp_{\phi,\lambda,\psi_1}\coprod\cdots\coprod \pp_{\phi,\lambda,\psi_r},
$$
where $\pp_{\phi,\lambda,\psi}$ is the subset of $\pp_{\phi,\lambda}$ formed by all prime ideals $\p$ such that $R_{\phi,\lambda}(F_\p)$ is a power of $\psi$ in $\F_{p,\phi}[y]$.

\subsection{The $p$-regularity condition}\label{subsecpReg}

\begin{definition}\label{pregular}
Let $\phi\in\Z_p[x]$ be a monic polynomial,  irreducible modulo $p$. We say that $f$ is \emph{$\phi$-regular} if for every side of $N_{\phi}^-(f)$, the residual polynomial attached to the side is squarefree.

Choose monic polynomials $\phi_1,\dots,\phi_t\in\Z_p[x]$ whose reduction modulo $p$ are the different irreducible factors of $f$ modulo $p$.
We say that $f$ is $p$-regular with respect to this choice if $f$ is $\phi_i$-regular for every $1\le i\le t$.
\end{definition}

If $f$ is $p$-regular, Theorems \ref{Diss2} and \ref{Diss3} provide the complete factorization of $f$ into a product of irreducible polynomials in $\Z_p[x]$, or equivalently, the decomposition of $p$ into a product of prime ideals of $K$. 

Moreover, in the $p$-regular case the $p$-index of $f$ is also determined by the shape of the different $\phi$-Newton polygons.

\begin{definition}\label{phindex}
The \emph{$\phi$-index of $f$} is $\deg \phi$ times the number of points with integer coordinates that lie below or on the polygon $N_{\phi}^-(f)$, strictly above the horizontal axis, and strictly beyond the vertical axis. We denote this number by $\ \ind_{\phi}(f)$.
\end{definition}

\begin{theorem}\label{index}
Let $\ind_p(f):=v_p\left(\left(\Z_K\colon\Z[\t]\right)\right)$.
With the above notation, $\ind_p(f)\ge\ind_{\phi_1}(f)+\cdots+\ind_{\phi_t}(f)$, and
equality holds if $f$ is $p$-regular.\hfill{$\Box$}
\end{theorem}

\section{Squarefree decomposition of polynomials modulo $N$}\label{secSFD}
\subsection{Squarefree decomposition in $(\Z/N\Z)[x]$}
Let $N>1$ be an integer and denote $A=\Z/N\Z$. We indicate simply with a bar the homomorphisms of reduction modulo $N$:
$$\hphantom{m}^{\overline{\hphantom{m}}}\,\colon\Z\lra A,\qquad \hphantom{m}^{\overline{\hphantom{m}}}\,\colon\Z[x]\lra A[x].$$

Also, for any prime number $p$ dividing $N$ we denote by the same symbol:
$$\rd_p\colon A\lra \F_p,\quad \rd_p\colon A[x]\lra \F_p[x]$$ the homomorphisms of reduction modulo $p$.

For an arbitrary $a\in A$ we define $\gcd(a,N)\in\Z_{>0}$ to be the unique positive divisor $m$ of $N$ such that $a$ is equal to $\overline{m}$ times a unit in $A$. 

Our first aim is to show that, under certain natural assumptions, there is a standard squarefree decomposition in the polynomial ring $\ax$.

\begin{definition}
A polynomial $g\in \ax$ is said to be \emph{almost-monic} if its leading coefficient is a unit in $A$.    
\end{definition}

If $g\in \ax$ is almost-monic, we may consider a $\op{Quotrem}$ routine:
$$
q,r=\op{Quotrem}(h,g)
$$
which for an arbitrary $h\in \ax$ computes $q,r\in \ax$ such that $h=gq+r$ and $\deg r<\deg g$. Clearly, 
\begin{equation}\label{qr}
 h\ax+g\ax=r\ax+g\ax.
\end{equation}

Also, an almost-monic  $g$ is a \emph{minimal} polynomial:
\begin{equation}\label{minimal}
\deg g=\mn\{\deg h\mid h\in g\ax, h\ne0\}. 
\end{equation}

In particular, for an almost-monic $g$ of positive degree, the chain of ideals generated by the powers of $g$ is strictly decreasing:
$$
\ax\supsetneq g\ax\supsetneq g^2\ax\supsetneq \cdots
$$
Hence, it makes sense to consider a function
$$
\ord_g\colon \ax\longrightarrow \Z_{\ge0}
$$
by defining $\ord_g(h)=k$ if $h\in g^k\ax$ but $h\not \in g^{k+1}\ax$. 

We may define a $\gcd$ routine for polynomials in $\ax$, with ``hooks" to detect a factorization of $N$. 

\Algo{GCD$_0$}{gcd0}
{$f,g\in\ax$ 
}
{Either a proper divisor of $N$, or a monic $d\!=\!\gcd_0(f,g)\!\in\!\ax$ such that $f\ax+g\ax=d\ax$. }
{
\begin{enumerate}
\item \quad while $g\ne0$ 
\item \qquad\quad $a\gets$ leading coefficient of $g$, \ $b\gets \gcd(a,N)$
\item \qquad\quad if $b\ne 1$ then return $b$ \ else \ $g\gets a^{-1}g$
\item \qquad\quad $q,\,r=\op{Quotrem}(f,g)$
\item \qquad\quad $f\gets g$, \ $g\gets r$
\item \quad return $f$
\end{enumerate}
}

Note that the identity $f\ax+g\ax=d\ax$ is an immediate consequence of (\ref{qr}) applied to each division with remainder. From this identity we deduce the following fundamental property.

\begin{lemma}\label{gcd}
Suppose the $\gcd_0$ routine does not factorize $N$ and outputs a polynomial $d\in \ax$. Then, for any prime divisor $p$ of $N$, the polynomial $\rd_p(d)$ is the greatest common divisor of $\rd_p(f)$ and $\rd_p(g)$ in $\F_p[x]$.    
\end{lemma}

With this $\gcd_0$ routine in hand, we can mimic the standard squarefree decomposition routine \cite[\S20.3]{shoup}, for polynomials of not too large degree. Let us be precise about the meaning of ``squarefree".

\begin{definition}
An almost-monic polynomial $g\in \ax$ is said to be \emph{squarefree} if $\rd_p(g)\in\F_p[x]$ is squarefree for all prime divisors $p$ of $N$.  
\end{definition}

\Algo{SFD$_0$}{sfd0}
{A monic $f\in\ax$ with $\deg f<p$ for all prime divisors $p$ of $N$. 
}
{Either a proper divisor of $N$ or a list $(g_1,\ell_1),\dots,(g_m,\ell_m)$, where $g_1,\dots,g_m\in\ax$ are monic, squarefree, pairwise coprime polynomials, and the integers $0\le \ell_1<\cdots <\ell_m$ satisfy $f=g_1^{\ell_1}\cdots g_m^{\ell_m}$. }
{
\begin{enumerate}
\item \quad $L\gets [\;]$ 
\item \quad $j\gets 1$, \ $g\gets f/\gcd_0(f,f')$
\item \quad while $f\ne 1$ do
\item \qquad\quad $f\gets f/g$, \ $h\gets\gcd_0(f,g)$, \ $t\gets g/h$
\item \qquad\quad if $t\ne 1$ then append $(t,j)$ to $L$
\item \qquad\quad $g\gets h$, \ $j\gets j+1$
\item \quad return $L$
\end{enumerate}
}

Of course, although not specifically indicated, after every call to $\gcd_0$ the routine ends if we find a proper divisor of $N$. 

As an immediate consequence of Lemma \ref{gcd}, we get a similar statement for this routine.

\begin{lemma}\label{sfd}
Suppose the SFD$_0$ routine does not factorize $N$ and outputs a list of pairs $(g_1,\ell_1),\dots,(g_m,\ell_m)$. Then, for any prime divisor $p$ of $N$, the list $(\rd_p(g_1),\ell_1),\dots,(\rd_p(g_m),\ell_m)$ is the canonical squarefree decomposition of $\rd_p(f)$ in $\F_p[x]$.    
\end{lemma}

In particular, this result justifies that the output polynomials $g_1,\dots,g_m\in\ax$ are squarefree and pairwise coprime.

\begin{remark}\rm
The condition $\deg f<p$ for all $p\mid N$ is quite reasonable for the construction of a global integral basis in a number field. In this context, $f$ is the reduction modulo $N$ of the defining polynomial of the number field and $N$ is a positive divisor of the discriminant. We may first remove from the discriminant all prime factors $p\le \deg f$ and then proceed with $N$ equal to the remaining factor.  
\end{remark}

\subsection{Squarefree decomposition of polynomials with coefficients in a finite extension of $A$}

Let us fix a monic squarefree polynomial $g\in\ax$. Consider the finite $A$-algebra
$A_1:=A[x]/(g)$. Clearly, $g\ax\cap A=0$, so that the natural map $A\to A_1$ is injective.

We want to describe a squarefree decomposition routine for polynomials
with coefficients in $A_1$.

A polynomial in $\a1x$ is said to be \emph{almost-monic} if it has a unitary leading coefficient. Almost-monic polynomials in $\a1x$ have completely analogous properties as those mentioned in the last section for polynomials in $\ax$.

Let $p$ be a prime factor of $N$. The $A$-algebra $A_1/pA_1$ is now a product of finite fields. In fact, if $\rd_p(g)=\varphi_1\cdots \varphi_t$ is the factorization of $\rd_p(g)$ into a product of monic ireducible polynomials in $\F_p[x]$, we have
$$
A_1/pA_1\,\simeq\, \F_p[x]/(\varphi_1) \times \cdots \times \F_p[x]/(\varphi_t). 
$$

In other words, the maximal ideals of $A_1$ are of the form $(p,\phi) \md{g}$, where $p$ is a prime divisor of $N$ and the reduction of $\phi\in \ax$ modulo $p$ is an irreducible factor of $\rd_p(g)$. \medskip

\noindent{\bf Notation.} 
For any maximal ideal $\m$ of $A_1$ we shall denote $\F_\m:=A_1/\m$ the corresponding finite field, and 
$$\rd_\m\colon A_1\lra \F_\m,\qquad  \rd_\m\colon \a1x\lra \F_\m[x],
$$
the homomorphisms of reduction modulo $\m$.
\medskip

Let us now discuss how to detect units in $A_1$.
Let $\alpha\in A_1$ be a non-zero element. Let $a\in \ax$ be a polynomial whose class modulo $g$ is $\alpha$. If we apply the routine GCD$_0$ to $a$ and $g$, the output has three possibilities:\medskip

(1) $\gcd_0(a,g)=1$,

(2) $\gcd_0(a,g)=b$, with $b\in\ax$ monic of positive degree,

(3) a factorization of $N$ has been detected.\medskip

In case (1), $\alpha$ is a unit in $A_1$, and $\alpha^{-1}$ may be computed from a B\'ezout identity $ra+sg=1$ in $\ax$, which may be obtained from an extended $\gcd_0$ implementation. 

In case (2), $b$ is a proper factor of $g$ in $\ax$, because $\alpha\ne0$.

Therefore, if $\alpha$ is not a unit, we gain relevant information about $N$ or $g$. 
This facilitates the design of a $\gcd$ routine with hooks for polynomials in $\a1x$, in the spirit of Algorithm \ref{gcd0}. 

\Algo{GCD$_1$}{gcd1}
{$f,g\in\a1x$ 
}
{Either a proper divisor of $N$, or a proper factor of $g$, or a monic $d=\gcd_1(f,g)\in\a1x$ such that $f\a1x+g\a1x=d\a1x$. }
{
\begin{enumerate}
\item \quad while $g\ne0$ 
\item \qquad\quad $a\gets$ leading coefficient of $g$, \ $b\gets \gcd_0(a,g)$
\item \qquad\quad if $b\ne 1$ then return $b$ \ else \ $g\gets a^{-1}g$
\item \qquad\quad $q,\,r=\op{Quotrem}(f,g)$
\item \qquad\quad $f\gets g$, \ $g\gets r$
\item \quad return $f$
\end{enumerate}
}

Clearly, for any maximal ideal $\m$ of $A_1$, the homomorphism $\rd_\m$ applied to each step of Algorithm \ref{gcd1} yields the standard $\gcd$ algorithm in $\F_\m[x]$. 

\begin{lemma}\label{gcd1good}
Suppose the $\gcd_1$ routine does not factorize $N$ nor $g$, and outputs a polynomial $d\in \a1x$. Then, for any maximal ideal $\m$ of $A_1$, the polynomial $\rd_\m(d)$ is the  greatest common divisor of $\rd_\m(f)$ and $\rd_\m(g)$ in $\F_\m[x]$.    
\end{lemma}

\begin{definition}
A almost-monic polynomial $h\in A_1[x]$ is said to be squarefree if $\rd_\m(h)\in\F_\m[x]$ is squarefree for all maximal ideals $\m$ of $A_1$. 
\end{definition}

We obtain an algorithm SFD$_1$ to compute squarefree decomposition in $\a1x$, just by replacing $A$ with $A_1$ and $\gcd_0$ with $\gcd_1$ in Algorithm \ref{sfd0}. 

\begin{lemma}\label{sf1}
Suppose the SFD$_1$ routine does not factorize $N$ nor $g$ and outputs a list of pairs $(g_1,\ell_1),\dots,(g_k,\ell_k)$. Then, for any maximal ideal $\m$ of $A_1$, the list $(\rd_\m(g_1),\ell_1),\dots,(\rd_\m(g_k),\ell_k)$ is the canonical squarefree decomposition of $\rd_\m(f)$ in $\F_\m[x]$.    
\end{lemma}

\section{Newton polygons modulo $N$}\label{secNewton}
We fix from now on a monic irreducible polynomial $f\in\Z[x]$ of degree $n>1$ and an integer $N>1$ whose prime divisors are all greater than $n$.

Let $K=\Q(\t)$ be the number field generated by a root $\t\in\qb$ of $f$, and denote by $\Z_K$ the ring of integers of $K$. Let $\pp_N\subset\op{Spec}(\Z_K)$ be the set of prime ideals of $\Z_K$ dividing $N$.

\subsection{$N$-adic valuation}
Let us consider the following $N$-adic valuation routine with hooks.

\Algo{$N$-adic valuation $v_N$}{vN}
{$a\in \Z$, $a\ne0$. 
}
{Either a proper divisor of $N$, or a non-negative integer $k=v_N(a)$ such that $a=N^kb$ with $b$ coprime to $N$. }
{
\begin{enumerate}
\item \quad $q\gets a$, \ $r\gets 0$, \ $\op{value}\gets -1$  
\item \quad while $r=0$ 
\item \qquad\quad $\op{value}\gets\op{value}+1$
\item \qquad\quad $q,\,r=\op{Quotrem}(q,N)$
\item \quad $d\gets \gcd(r,N)$
\item \quad if $d\ne1$ then return $d$ else return $\op{value}$
\end{enumerate}
}

\begin{lemma}\label{vNgood}
Suppose the $v_N$ routine does not factorize $N$ and outputs $v_N(a)=k$. Then,  $v_p(a)=kv_p(N)$  for all prime divisors $p$ of $N$.    
\end{lemma}

We agree that $v_N(0)=\infty$. Also, we extend this function (which is not a valuation) to a function on $\Z[x]$, in the usual way:
$$
v_N(a_0+a_1x+\cdots+a_sx^s+\cdots)=\mn_{0\le s}\{v_N(a_s)\}.
$$

\subsection{Newton polygons and residual polynomials}
Consider monic po\-lynomials $g_1,\dots,g_m\in\Z[x]$ such that
$$
f\equiv g_1^{\ell_1}\cdots g_m^{\ell_m} \ \md{N}
$$
is the canonical squarefree decomposition of $f$ modulo $N$. 

Since these polynomials are pairwise coprime modulo $N$, we have an analogous of the first dissection:
$$
\pp_N=\coprod_{i=1}^m\pp_{g_i},\quad \pp_g=\left\{\p\in\pp_N\mid v_\p(g(\t))>0\right\}.
$$

Let us fix one of these polynomials; say $g=g_i$ for some $1\le i\le m$, and denote $\ell:=\ell_i=\ord_{\overline{g}}\left(\overline{f}\right)$.

Our polynomial $f$ admits a unique $g$-expansion:
\begin{equation}\label{phiadicN}
f=a_0+a_1g+\cdots+a_rg^r,
\end{equation}
with $a_i\in\Z[x]$ having $\deg a_i<\deg g$. For any coefficient $a_i$ we compute the $N$-adic value $u_i=v_N(a_i)\in\Z\cup\{\infty\}$.

\begin{definition}
The $g$-Newton polygon of $f$ is the lower convex hull of the set of all points $(i,u_i)$, $u_i<\infty$, in the Euclidian plane.
We denote this open convex polygon by $N_g(f)$.
\end{definition}

The \emph{length} of this polygon is by definition the abscissa $r$ of the last vertex. We denote it by  $\ell(N_g(f))$. The typical shape of this polygon is shown in Figure \ref{fig1}.

\begin{definition}
The polygon $N^-_g(f)$ determined by the sides of negative slope of $N_g(f)$ is called the \emph{principal $g$-polygon} of $f$. 
\end{definition}

\begin{lemma}\label{ell=ord}
The length of $N^-_g(f)$ is equal to $\ell=\ord_{\overline{g}}\left(\overline{f}\right)$. \end{lemma}

\begin{proof}
Let $b=a_0+a_1g+\cdots +a_{\ell-1}g^{\ell-1}$. Since $\overline{f}$ belongs to the ideal $\overline{g}^\ell\ax$, the element $\overline{b}$ belongs to this ideal too. Since $\deg b<\deg g^\ell$ and $\overline{g}$ is monic, we have $\overline{b}=0$ by (\ref{minimal}). This implies  $\overline{a}_0=\cdots=\overline{a}_{\ell-1}=0$, again because   $\overline{g}$ is monic.

In particular, $u_0,\dots,u_{\ell-1}>0$, and it remains only to show that $u_\ell=0$.
 Since $\overline{f}$ does not belong to the ideal $\overline{g}^{\ell+1}\ax$, we deduce that $\overline{a}_\ell\overline{g}^\ell$ does not belong to this ideal either. Hence, $\overline{a}_\ell\ne0$ and $u_\ell=0$.
\end{proof}

Let $A_1:=\ax/(g)=\Z[x]/(N,g)$ and denote by 
$ \rd_{N,g}\colon \Z[x]\longrightarrow A_1$
the homomorphism of reduction modulo $(N,g)$.
We attach to any integer abscissa $0\le i\le \ell$ the following \emph{residual coefficient} $c_i\in A_1$:
$$\as{1.6}
c_i=\left\{\begin{array}{ll}
0,&\mbox{ if $(i,u_i)$ lies above $N_g(f)$, or }u_i=\infty,\\\rd_{N,g}\left(\dfrac{a_i}{N^{u_i}}\right),&\mbox{ if $(i,u_i)$ lies on }N_g(f).
\end{array}
\right.
$$ Note that $c_i$ is always nonzero in the latter case, because $\deg a_i<\deg g$. 

Let $S$ be one of the sides of $N^-_g(f)$ with slope $-\lambda=-h/e$, where $h,e$ are positive coprime integers. Define the \emph{degree} $d(S)$ of $S$ as in section \ref{subsec3rd}.

\begin{definition}
Let $s$ be the initial abscissa of $S$, and let $d=d(S)$. The \emph{residual polynomial} attached to $S$ (or to $\lambda$) is defined as:
$$
R_{g,\lambda}(f)=c_s+c_{s+e}\,y+\cdots+c_{s+(d-1)e}\,y^{d-1}+c_{s+de}\,y^d\in A_1[y].
$$
\end{definition}

Note that $c_s$ and $c_{s+de}$ are always nonzero, so that the residual polynomial has degree $d$ and is never divisible by $y$.

Ideally, we would like these polygons and residual polynomials to have analogous properties to those stated in Theorems \ref{Diss2} and \ref{Diss3}, but this is not true. However, we still face a win-win situation: when one of these objects fails to have the properties we need for the final computation of an $N$-integral basis, then it yields a factorization of $N$ or $\overline{g}$. The next section is devoted to the discussion of this phenomenon.

\subsection{Admissible $\phi$-expansions}
Let $p$ be one of the prime factors of $N$, and denote the homomorphisms of reduction modulo $p$ by the same symbol:
$$
\rd_p\colon \Z[x]\lra \F_p[x],\qquad \rd_p\colon \Z_p[x]\lra \F_p[x].
$$

Let $\rd_p(g)=\varphi_1\cdots \varphi_t$ be the factorization of $\rd_p(g)$ into a product of monic irreducible polynomials in $\F_p[x]$. By Hensel's lemma, $g$ splits in $\Z_p[x]$ into a product of $t$ monic irreducible polynomials:
$$
g=\phi_1\cdots \phi_t,\qquad \rd_p(\phi_i)=\varphi_i, \ 1\le i \le t.
$$

\begin{lemma}\label{length}
Let $p$ be a prime divisor of $N$ and let $\phi$ be a $p$-adic irreducible factor of $g$.  Let $\pp_\phi=\coprod_\lambda\pp_{\phi,\lambda}$ be the partition of $\pp_\phi$ determined by the slopes of the different sides of $N^-_\phi(f)$ (section \ref{subsec2nd}). 
\begin{enumerate}
\item $\ell(N_\phi^-(f))=\ell(N_g^-(f))$. 
\item For any $\p\in\pp_{\phi,\lambda}$ we have $v_\p(g(\t))=e(\p/p)\lambda$. 
\end{enumerate}
\end{lemma}

\begin{proof}
We saw in section \ref{subsec2nd} that $\ell(N_\phi^-(f))=\ord_{\rd_p(\phi)}\rd_p(f)=\ell$. Since $\rd_p(g)$ is squarefree and the polynomials $g_1,\dots,g_m$ are pairwise coprime, this length coincides with $\ord_{\overline{g}}\overline{f}$, which is equal to $\ell(N_g^-(f))$ by Lemma \ref{ell=ord}. This proves item (1).

Take $\p\in\pp_{\phi,\lambda}$. By Theorem \ref{Diss2}, we have $v_p(\phi(\iota_\p(\t)))=\lambda$. Since all $p$-adic factors of $g$ are pairwise coprime modulo $p$, we have $v_p(\phi_j(\iota_\p(\t)))=0$ for all $\phi_j\ne\phi$. Hence, $v_p(g(\iota_\p(\t)))=\lambda$. This proves item (2).
\end{proof}

Let us fix one of these $p$-adic irreducible factors; say $\phi=\phi_j$. 

Let $\Phi\in\Z[x]$ be any integer poynomial congruent to $\phi$ modulo $p$, and consider the ideal $(p,\Phi)$. The ideal $\m=\rd_{N,g}(p,\Phi)$ is a maximal ideal of $A_1$. We abuse of language and denote this ideal simply as $\m=(p,\phi)$.

Consider the canonical $\phi$-expansion of $f$:
$$
f=b_0+b_1\phi+\cdots+b_\ell\phi^\ell+ \cdots,\qquad \deg b_s<\deg \phi,\,\forall s\ge0.
$$
Consider now another $\phi$-expansion, not necessarily the canonical one:
\begin{equation}\label{phdev}
f=b'_0+b'_1\phi+\cdots+b'_r\phi^r+ \cdots.
\end{equation}

Take $u'_i=v_p(b'_i)$, for all $i\ge0$, and let $N'$ be the
principal polygon of the set of points $(i,u'_i)$. To any abscissa $0\le i\le \ell(N')$ we attach a residual coefficient as before:%\footnote{Note that $\phi$ cannot divide $f$. If $f$ and $g$ had a common $p$-adic factor, then they would have a rational factor too and this is impossible because $f$ is irreducible and $\deg g<\deg f$.}:
 $$\as{1.2}
c'_i=\left\{\begin{array}{ll}
0,&\mbox{ if $(i,u'_i)$ lies above }N', \mbox{ or }u'_i=\infty,\\\rd_{p,\phi}\left(b'_i/p^{u'_i}\right),&\mbox{ if $(i,u'_i)$ lies on }N'.
\end{array}
\right.
$$ 
For the points $(i,u'_i)$ lying on $N'$ we can now have $c'_i=0$, because $b'_i$ could be divisible by $\phi$.

Finally, for any side $S'$ of $N'$ of negative slope $-\lambda=-h/e$ we can define the residual polynomial 
$$
R'_{\phi,\lambda}(f):=c'_{s'}+c'_{s'+e}\,y+\cdots+c'_{s'+(d'-1)e}\,y^{d'-1}+c'_{s'+d'e}\,y^{d'}\in\fph[y],
$$where  $d'=d(S')$ and $s'$ is the abscissa of the left endpoint of $S'$.

\begin{definition}\label{adm}
We say that the $\phi$-expansion (\ref{phdev}) is \emph{admissible} if $c'_s\ne0$
for each abs\-cissa $s$ of a vertex of $N'$.
\end{definition}

Admissible expansions yield the same principal polygon and the same residual polynomials. This was proved in \cite[Lem. 1.12]{GMN}.

\begin{lemma}\label{admissible}
If a $\phi$-expansion is admissible, then
$N'=\npp(f)$ and $c'_i=c_i$ for all abscissas $0\le i\le \ell(N')$. In particular, for any negative slope $-\lambda$ we have $R'_{\phi,\lambda}(f)=R_{\phi,\lambda}(f)$.
\end{lemma}

\begin{theorem}\label{equalpolygons}
Suppose that the construction of $N_g^-(f)$ does not fail, and for each vertex $(s,u_s)$ of $N_g^-(f)$ the $\gcd_0$ routine does not fail and outputs $\gcd_0\left(\overline{a_s/N^{u_s}},\overline{g}\right)=1$. 

Take a prime divisor $p$ of $N$ and a $p$-adic irreducible factor $\phi$ of $g$. Let $\m=(p,\phi)$ be the corresponding maximal ideal of $A_1$. Then,

\begin{enumerate}
\item The canonical $g$-expansion of $f$ is an admissible $\phi$-expansion. 
\item $N_\phi^-(f)=E_\rho(N_g^-(f))$, where  $\rho=v_p(N)$ and $E_\rho\colon \R^2\lra \R^2$ is the affine plane transformation given by $E_\rho(x,y)=(x,\rho y)$.
\item For any slope $-\lambda=-h/e$ of $N_g^-(f)$, let $P=\rd_\m\left(R_{g,\lambda}(f)\right)\in\F_\m[y]$. There exist $\sigma,\tau\in \F_\m^*$ such that $R_{\phi,\rho\lambda}(f)(y)=\tau P(\sigma y^{\gcd(\rho,e)})$.
\end{enumerate}
\end{theorem}

\begin{proof}
Rewrite the canonical $g$-expansion of $f$ as a $\phi$-expansion:
$$%\begin{equation}\label{gphi}
f=a_0+a_1g+\cdots+a_\ell g^\ell+\cdots=b'_0+b'_1\phi+\cdots + b'_\ell\phi^\ell+\cdots,
$$%\end{equation}
where $b'_i=a_i\Phi^i$, for $\Phi:=g/\phi$. Let $N'$ be the principal Newton polygon determined by the cloud of points $(i,u'_i)$, where $u'_i=v_p(b'_i)$ for all $i$.

Since $\Phi$ is a monic polynomial we have $v_p(\Phi)=0$, so that $u'_i=v_p(a_i)$ for all $i$. Hence, $N'$ coincides with the principal Newton polygon associated with the canonical $g$-expansion of $f$, with respect to the $p$-adic valuation $v_p$. Since $v_p=\rho v_N$, this shows that
$N'=E_\rho(N_g^-(f))$.

In particular, the slopes of $N'$ are $-\rho\lambda$, for $-\lambda$ running on the slopes of $N_g^-(f)$. Also, the vertices of $N'$ and $N_g^-(f)$ have the same abscissas.

Let us now prove item (1). For any abscissa $s$ of a vertex of $N'$ we have $c'_s=\rd_{p,\phi}(b'_s/p^{u'_s})=\rd_{p,\phi}(a_s/p^{u'_s})\rd_{p,\phi}(\Phi^s)$. Clearly, $\rd_{p,\phi}(\Phi)\ne0$ because $\Phi$ and $\phi$ are coprime modulo $p$. On the other hand, $\gcd_0\left(\overline{a_s/N^{u_s}},\overline{g}\right)=1$ by our assumptions; thus, Lemma \ref{gcd} shows that the reduction modulo $p$ of $a_s/N^{u_s}$ and $g$ are coprime. Since $N^{u_s}=p^{u'_s}M$, with $p\nmid M$, we deduce that $\rd_{p,\phi}(a_s/p^{u'_s})\ne0$ as well. Therefore $c'_s\ne0$, which is the condition of admissibility. 

By Lemma \ref{admissible}, we have $N'=N_\phi^-(f)$ and $R'_{\phi,\rho\lambda}(f)=R_{\phi,\rho\lambda}(f)$ for all slopes $\rho\lambda$ of $N_\phi^-(f)$. This proves item (2).

Also, in order to prove item (3) we may compare $P:=\rd_\m\left(R_{g,\lambda}(f)\right)$ with $R'_{\phi,\rho\lambda}(f)$ instead of  $R_{\phi,\rho\lambda}(f)$.

The compatibility of the different reduction maps is a consequence of the commutativity of the following diagram:

\begin{center}
\setlength{\unitlength}{4mm}
\begin{picture}(10,5.6)
\put(0,5){$\Z[x]$}\put(2.5,5.4){\vector(1,0){3.4}}\put(7.5,5){$\Z_p[x]$}
\put(-2.2,2.6){$\rd_{N,g}$}\put(9,2.6){$\rd_{p,\phi}$}
\put(.6,0){$A_1$}\put(2.5,.4){\vector(1,0){3.5}}
\put(3.2,1){$\rd_\m$}\put(6.8,0){$\F_\m=\F_{p,\phi}$}
\put(8.5,4.2){\vector(0,-1){3}}\put(1,4.2){\vector(0,-1){3}}
\end{picture}
\end{center}

Let $s$ be the initial (left) abscissa of the side of slope $-\lambda$ of $N_g^-(f)$, and consider
$$
z=\rd_{p,\phi}(N/p^\rho),\quad w=\rd_{p,\phi}(\Phi),\quad
\tau=w^sz^{u_s},\quad \sigma=w^ez^{-h}.
$$

For $0\le j\le d(S)$, let $c_{s_j}\in A_1$ be the coefficient of the monomial of degree $j$ of $R_{g,\lambda}(f)$, where $s_j=s+je$. If $c_{s_j}\ne0$, then by definition, $c_{s_j}=\rd_{N,g}\left(a_{s_j}/N^{u_{s_j}}\right)$. By using $u_{s_j}=u_s-jh$, a straightforward computation shows  that
$$
\tau \sigma^j\rd_\m(c_{s_j})=\rd_{p,\phi}\left(b'_{s_j}/p^{u'_{s_j}}\right)=c'_{s_j}.
$$
Finally, in the Newton polygon $N'$, the slope $-\rho\lambda=-\rho h/e$ has least positive denominator $e'=e/k$, for $k=\gcd(e,\rho)$. Hence,  $c'_{s_j}=c'_{s+je}=c'_{s+jke'}$ is actually the coefficient of the monomial of degree $kj$
of $R'_{\phi,\rho\lambda}(f)$. This ends the proof of the third item. 
\end{proof}

Conclusions (2) and (3) of Theorem \ref{equalpolygons} are very strong. For sure, the polygons $N_{\phi_j}^-(f)$ have the same length, but could have a very different configuration of sides and slopes. It seems unlikely that all these Newton polygons coincide and coincide as well (apart from the eventual change of the value of $\rho$) with all polygons attached to all other prime divisors of $N$. This suggests that the condition $\gcd_0(\overline{a_s/N^{u_s}},\overline{g})=1$ for all vertices of $N_g^-(f)$ should fail relatively often, providing a factorization of $N$ or $\overline{g}$ with a reasonably high probability. 

\begin{corollary}\label{Diss2N}
Under the same assumptions, let us focus our attention on the set $\pp(p)$ of the prime ideals of $\Z_K$ dividing $p$:
$$
\pp_{p,g}:=\pp_g\cap\pp(p)=\{\p\in \pp(p)\mid v_\p(g(\t))>0\}.
$$
Then, we get a second dissection:
$$
\pp_{p,g}=\coprod\nolimits_{\lambda}\pp_{p,g,\lambda},\quad \pp_{p,g,\lambda}=\left\{\p\in\pp_{p,g}\mid v_\p(g(\t))=e(\p/p)\rho \lambda\right\},
$$
for $-\lambda$ running on the slopes of $N_g^-(f)$.
\end{corollary}

\begin{proof}
If $g=\phi_1\cdots\phi_t$ is the factorization of $g$ in $\Z_p[x]$, we have clearly
$\pp_{p,g}=\coprod_{i=1}^t\pp_{\phi_i}$. Now, since all Newton polygons $N_{\phi_i}^-(f)$ concide with $E_\rho(N_g^-(f))$, the classical second dissection for the prime $p$ is:
$$
\pp_{\phi_i}=\coprod\nolimits_{\lambda}\pp_{\phi_i,\rho\lambda},
$$
for $-\lambda$ running on the slopes of $N_g^-(f)$. Finally, item (2) of Lemma \ref{length}
shows that $\pp_{p,g,\lambda}=\bigcup_{i=1}^t\pp_{\phi_i,\rho\lambda}$.
\end{proof}

\begin{corollary}\label{sqf}
Under the same assumptions, $R_{g,\lambda}(f)\in A_1[y]$ is squarefree if and only if $R_{\phi,v_p(N)\lambda}(f)\in\F_\m[y]$ is squarefree, for all maximal ideals $\m=(p,\phi)$ of $A_1$. 
\end{corollary}

\begin{proof}
By definition, $R_{g,\lambda}(f)$ is squarefree if and only if $\rd_\m\left(R_{g,\lambda}(f)\right)$ is squarefree for all maximal ideals $\m$ of $A_1$. Now, a polynomial $P\in\F_\m[y]$ is squarefree if and only if it has a non-zero discriminant. This property is preserved if we transform $P$ into $\tau P(\sigma y)$, for $\sigma,\tau\in\F_\m^*$. Also, if $P(0)\ne0$, then squarefreeness is preserved if we transform $P$ into $P(y^k)$ for $k>0$. By construction, the residual polynomials have a non-zero constant term; hence, Theorem \ref{equalpolygons} shows that $\rd_\m\left(R_{g,\lambda}(f)\right)$ is squarefree if and only if $R_{\phi,v_p(N)\lambda}(f)$ is squarefree. 
\end{proof}

\subsection{The condition of $N$-regularity}\label{subsecNReg}
Suppose that the routine SFD$_0$ does not fail and outputs a squarefree decomposition of $\overline{f}$. Let us choose representatives $g_1,\dots, g_m\in\Z[x]$ of the squarefree factors of $\overline{f}$ in $A[x]$.

We say that $f$ is \emph{$N$-regular} with respect to these choices
if the following conditions are satisfied:

\begin{enumerate}
\item The construction of none of the Newton polygons $N_{g_i}^-(f)$ fails. 
\item For all $i$, and for each abscissa $s$ of a vertex of $N_{g_i}^-(f)$, the $\gcd_0$ routine does not fail and outputs $\gcd_0\left(\overline{a_{s,i}/N^{u_{s,i}}},\overline{g_i}\right)=1$, where $a_{s,i}$ is the $s$-th coefficient of the $g_i$-expansion of $f$ and $u_{s,i}=v_N(a_{i,s})$. 
\item For all $i$, and for each slope $-\lambda$ of $N_{g_i}^-(f)$, the routine SFD$_1$ does not fail and certificates that $R_{g_i,\lambda}(f)$ is squarefree. 
\end{enumerate}

The next result is an immediate consequence of  Theorem \ref{equalpolygons} and Corollary \ref{sqf}.

\begin{corollary}\label{regreg}
Supose $f$ is $N$-regular with respect to the choice of the polynomials $g_1,\dots,g_m$. Then, for all primes $p$ dividing $N$, our defining polynomial $f$ is $p$-regular with respect to the choices of all $p$-adic irreducible factors in $\Z_p[x]$ of all $g_i$, as representatives of the pairwise different irreducible factors of $f$ modulo $p$.
\end{corollary}

\section{Computation of an $N$-integral basis in the regular case}\label{secMain}
We keep with the notation of the preceding sections. 
The aim  of this section is to compute an $N$-integral basis, under the assumption that the defining polynomial $f$ of our number field $K$ is $N$-regular.

\begin{definition}
Let $\zp$ be the localization of $\Z$ at the prime ideal $p\Z$.   
A \emph{$p$-integral basis} of $\Z_K$ is a family $\alpha _1,\dots,\alpha _n\in\Z_K$ such that $\alpha _1\otimes1,\dots,\alpha _n\otimes1$ is a $\zp$-basis of $\Z_K\otimes_\Z\zp$.

We say that $\alpha _1,\dots,\alpha _n\in\Z_K$ is an \emph{$N$-integral basis} of $\Z_K$ if
it is a $p$-integral basis for all prime divisors $p$ of $N$.
\end{definition}

Suppose we choose representatives $g_1,\dots, g_m\in\Z[x]$ of the squarefree decomposition of $\overline{f}$ in $A[x]$ provided by the routine SFD$_0$. Thus,
$$
f\equiv g_1^{\ell_1}\dots\;g_m^{\ell_m} \md{N}, \quad 0\le \ell_1<\cdots<\ell_m.
$$

Consider the different $g_i$-expansions of $f$:
$$
f=a_{i,0}+a_{i,1}g_i+\cdots+a_{i,r_i}g_i^{r_i},\quad 1\le i\le m.  
$$

\begin{definition}
The \emph{quotients} attached to each $g_i$-expansion are, by definition, the different quotients $q_{i,1},\dots,q_{i,r_i}$ that are obtained along the computation of the coefficients of the expansion:
$$
f=g_iq_{i,1}+a_{i,0},\quad
q_{i,1}=g_iq_{i,2}+a_{i,1},\quad
\cdots\quad
q_{i,r_i}=g_i\cdot 0+a_{i,r_i}=a_{i,r_i}.
$$
Equivalently, $q_{i,j}$ is the quotient of the division of $f$ by $g_i^j$.
\end{definition}

Suppose that the construction of none of the Newton polygons $N_{g_i}^-(f)$ fails. By Lemma \ref{length}, $\ell(N_{g_i}^-(f))=\ell_i$ for all $i$. 
For any integer abscissa $0 \le j\le \ell_i$, let $y_{i,j}\in\Q$ be the ordinate of the point of $N_{g_i}^-(f)$ of abscissa $j$. For any $1\le i\le m$ fixed, these rational numbers form a strictly decreasing sequence, and $y_{i,\ell_i}=0$ (see Figure \ref{fig2}).

\begin{figure}
\begin{center}
\setlength{\unitlength}{5.mm}
\begin{picture}(10,7)
\put(6.85,-.2){$\bullet$}\put(2.85,1.85){$\bullet$}\put(-.15,4.85){$\bullet$}
\put(-1,0){\line(1,0){11}}\put(0,-1){\line(0,1){7}}
\put(0,5){\line(1,-1){3}}\put(0,5.03){\line(1,-1){3}}
\put(3,2){\line(2,-1){4}}\put(3,2.03){\line(2,-1){4}}
\put(7,-.7){\begin{footnotesize}$\ell_i$\end{footnotesize}}
\put(.85,-.7){\begin{footnotesize}$1$\end{footnotesize}}
\put(1.85,-.7){\begin{footnotesize}$2$\end{footnotesize}}
\put(5.1,-.7){\begin{footnotesize}$\ell_i-1$\end{footnotesize}}
\put(-.4,-.7){\begin{footnotesize}$0$\end{footnotesize}}
\put(2.6,4.6){\begin{footnotesize}$N_{g_i}^-(f)$\end{footnotesize}}
\put(1,-.2){\line(0,1){4.2}}\put(2,-.2){\line(0,1){3.2}}
\put(6,-.2){\line(0,1){.7}}
\multiput(-.1,4)(.25,0){5}{\hbox to 2pt{\hrulefill }}
\multiput(-.1,3)(.25,0){9}{\hbox to 2pt{\hrulefill }}
\multiput(-.1,.55)(.25,0){25}{\hbox to 2pt{\hrulefill }}
\put(-1.3,3.9){\begin{footnotesize}$y_{i,1}$\end{footnotesize}}
\put(-1.3,2.9){\begin{footnotesize}$y_{i,2}$\end{footnotesize}}
\put(-2,.4){\begin{footnotesize}$y_{i,\ell_i-1}$\end{footnotesize}}
\end{picture}\caption{}\label{fig2}
\end{center}
\end{figure}

Note that the sum $\lfloor y_{i,1}\rfloor+\cdots+\lfloor y_{i,\ell_i}\rfloor$ coincides with the number of points of integer coordinates in the region delimited by the polygon and the axes.

\begin{theorem}\label{main}
Suppose that $f$ is $N$-regular and consider the following elements in the order $\Z[\t]$:
\begin{equation}\label{basis}
\alpha_{i,j,k}=q_{i,j}(\t)\,\t^k,\qquad 1\le i\le m,\ 1\le j\le\ell_i,\ 0\le k<\deg g_i.
\end{equation}

Then the family of all $\alpha_{i,j,k}/N^{\lfloor y_{i,j}\rfloor}$ is an $N$-integral
basis of $\Z_K$ if and only if at least one of the following two conditions is satisfied.
\begin{enumerate}
\item[(a)] $N$ is squarefree.
\item[(b)] All slopes of all Newton polygons $N_g^-(f)$ are integers.
\end{enumerate}

\end{theorem}

This is the main theorem of the paper. The proof will follow the arguments in \cite[\S2]{np4}, where the theorem was proved for $N$ prime. 

\begin{proposition}\label{denominator}
For all $1\le i\le m$ and all $1\le j<\ell_i$, the element $q_{i,j}(\t)/N^{\lfloor y_{i,j}\rfloor}$ belongs to $\Z_K$.
\end{proposition}

\begin{proof}
We fix an index $1\le i\le m$, and we denote $g=g_i$, $\ell=\ell_i$, $q_j=q_{i,j}$, $y_j=y_{i,j}$. Also, let $f=a_0+a_1g+\cdots+a_rg^r$ be the $g$-expansion of $f$.

By definition, for all $1\le j\le r$ we have:  
\begin{equation}\label{quotient}
q_j=a_j+a_{j+1}g+\cdots+a_rg^{r-j},
\end{equation}
\begin{equation}\label{division}
f=r_j+q_j\,g^j, \mbox{ where }
r_j=a_0+a_1g+\cdots+a_{j-1}g^{j-1}.
\end{equation}

Take an index $1\le j<\ell$. We must show that for all prime divisors $p$ of $N$ and all prime ideals $\p$ dividing $p$, we have $v_{\p}(q_j(\t))\ge
e(\p/p)\rho y_j$, where $\rho=v_p(N)$. From now on, we consider $p$ (and $\rho$) fixed.

Let $-\lambda_1<\cdots<-\lambda_k$ be the slopes of the different sides of
$N_g^-(f)$. Recall the second dissection of Corollary \ref{Diss2N}. The set
$\pp_{p,g}$ splits into the disjoint union of the $k$ subsets
$$\pp_{p,g,\lambda_s}=\{\p\in\pp_{p,g}\mid v_{\p}(g(\t))=e(\p/p)\rho\lambda_s\}, \quad 1\le s\le k.$$
 Let $1\le z\le k$ be the greatest
index such that the projection of the side of slope $-\lambda_z$ to the horizontal axis contains the abscissa $j$.

Denote $u_s=v_N(a_s)=\rho v_p(a_s)$, for all $s$. Suppose first that
$\p\in\pp_{p,g,\lambda_q}$ for some $q\le z$. In this case, for all $j\le s$ we have:
\begin{align*}
v_{\p}(a_s(\t)g(\t)^{s-j})/(e(\p/p)\rho)\ge&\ u_s+(s-j)\lambda_q\ge u_s+(s-j)\lambda_z\\\ge&\ y_s+(s-j)\lambda_z  \ge y_j,
\end{align*}the last inequality by the convexity of the Newton polygon.
This shows that $v_{\p}(q_j(\t))\ge e(\p/p)\rho y_j$, because all summands of (\ref{quotient}) have this property.

Suppose now that either $\p\not\in\pp_{p,g}$ or $\p\in \pp_{g,\lambda_q}$ for
some $q>z$; that is, $v_{\p}(g(\t))=e(\p/p)\rho\mu$ for some
$\mu<\lambda_z$ ($\mu=0$ if $\p\not\in\pp_{p,g}$). In this case, we use the identities in (\ref{division}),
which imply, again:
\begin{align*}
v_\p(q_j(\t))/(e(\p/p)\rho)=&\ v_\p(r_j(\t))/(e(\p/p)\rho)-j\mu\\\ge&\ \mn_{0\le s<j}\{v_\p(a_s(\t)g(\t)^s)/(e(\p/p)\rho)\}-j\mu\\ 
=&\ v_\p(a_{s_0}(\t)g(\t)^{s_0})/(e(\p/p)\rho)-j\mu\ge u_{s_0}-(j-s_0)\mu\\
>&\ u_{s_0}-(j-s_0)\lambda_z\ge y_{s_0}-(j-s_0)\lambda_z\ge\ y_j,
\end{align*}the last inequality by the convexity of the Newton polygon.
\end{proof}

\begin{lemma}\label{ordphi}
Let $p$ be a prime divisor of $N$ and $\varphi\in\F_p[x]$ a monic irreducible factor of $\rd_p(f)$. Then, 
$$
\ord_\varphi(\rd_p(q_{i,j}))=\left\{\begin{array}{ll}
 \ell_i-j,&\qquad \mbox{if }\varphi\mid \rd_p(g_i),\\
\ell_q,&\qquad \mbox{if }\varphi\mid \rd_p(g_q),\ q\ne i,
\end{array}
\right.
$$
for all $1\le i\le m$, $1\le j\le \ell_i$.
\end{lemma}

\begin{proof}
For commodity we indicate reduction modulo $p$ simply with a bar.
Let $1\le q\le m$ be the unique index such that $\varphi\mid\overline{g}_q$, so that $\ord_\varphi(\overline{f})=\ell_q$. As we saw in section \ref{subsec2nd}, $\overline{a}_{i,\ell_i}\ne0$ and $\ell_i$ is the least index with this property. Hence, 
$\overline{f}=\overline{q}_{i,\ell_i}\overline{g}_i^{\ell_i}$.
We deduce that
$$
\ord_\varphi(\overline{q}_{i,\ell_i})=\left\{\begin{array}{ll}
 0,&\qquad \mbox{if }q=i,\\
\ell_q,&\qquad \mbox{if }q\ne i.
\end{array}
\right.
$$
Now, for all $j<\ell_i$, the identity (\ref{quotient}) shows that $\overline{q}_{i,j}=\overline{q}_{i,\ell_i}\,\overline{g}_i^{\ell_i-j}$, and this ends the proof of the lemma.
\end{proof}

\begin{lemma}\label{numerators}
Let $p$ be a prime divisor of $N$. 
The elements $\alpha_{i,j,k}\in\Z[\t]$ given in (\ref{basis}) form a $\zp$-basis of $\zp[\t]$.
\end{lemma}

\begin{proof}
Let us show first that the $n=\ell_1\deg g_1+\cdots+\ell_m\deg g_m$ polynomials $q_{i,j}x^k$ are linearly independent modulo $p$. Denote by $Q_{i,j,k}=\overline{q}_{i,j}\overline{x}^k$ their reduction modulo $p$, and suppose that  
\begin{equation}\label{li}
\sum\nolimits_{i,j,k}a_{i,j,k}Q_{i,j,k}=0,
\end{equation}
for some constants $a_{i,j,k}\in\F_p$. Consider the following polynomials in $\F_p[x]$: 
\begin{equation}\label{aij}
A_{i,j}=\sum\nolimits_{0\le k<\deg g_i}a_{i,j,k}Q_{i,j,k},\qquad B_{i,j}=\sum\nolimits_{0\le k<\deg g_i}a_{i,j,k}x^k,
\end{equation}
so that $A_{i,j}=\overline{q}_{i,j}B_{i,j}$ for all $i,j$.
Now, the equality (\ref{li}) is equivalent to $\sum_{i,j}A_{i,j}=0$. Let us show that this implies $A_{i,j}=0$ for all $i,j$. In fact, let us assume that the indices $1\le i\le m$ are ordered so that $\ell_1\ge \cdots\ge \ell_m$. Suppose that $A_{i_0,j_0}\ne0$ for a maximal pair $(i_0,j_0)$ with respect to the lexicographical order. Since $\deg B_{i_0,j_0}<\deg g_{i_0}$, there exists an irreducible factor $\varphi$ of $\overline{g}_{i_0}$ in $\F_p[x]$ such that $\varphi\nmid B_{i_0,j_0}$. By Lemma \ref{ordphi}, $\ord_\varphi(A_{i_0,j_0})=\ell_{i_0}-j_0$.

Now, for any pair $(i,j)$ with $A_{i,j}\ne 0$, we have  either $i< i_0$, or $i=i_0$, $j<j_0$, by the maximality of $(i_0,j_0)$.  Lemma \ref{ordphi} shows in both cases that  $\ord_\varphi(A_{i_0,j_0})<\ord_\varphi(A_{i,j})$. This implies $\ord_{\varphi}\left(\sum_{i,j}A_{i,j}\right)=\ell_{i_0}-j_0$, which is a contradiction. 

Thus, $B_{i,j}=0$ for all $i,j$, and this implies $a_{i,j,k}=0$ for all $i,j,k$. Therefore, the polynomials  
$Q_{i,j,k}$ are $\F_p$-linearly independent.

In particular, since all polynomials $q_{i,j}x^k$ have degree less than $n$, the integral elements $\alpha_{i,j,k}$ are $\Z$-linearly independent. Finally, they generate $\zp[\t]$ as a $\zp$-module by Nakayama's lemma. 
\end{proof}

\noindent{\bf Proof of Theorem \ref{main}.}
Consider the following  $\Z$-submodules:
$$
M=\gen{\alpha_{i,j,k}}_{\Z}\subset \Z[\t],\quad
M'=\gen{\alpha_{i,j,k}/N^{\lfloor y_{i,j}\rfloor}}_{\Z}\subset \Z_K.
$$

By Lemma \ref{numerators}, the elements $\alpha_{i,j,k}/N^{\lfloor y_{i,j}\rfloor}$ are $\Z$-linearly independent. Thus, we need only to prove that they generate  $\Z_K\otimes \zp$ as a $\zp$-module, for all prime divisors $p$ of $N$. 

For any such $p$, consider the following chain of free $\zp$-modules
$$
\zp[\t]= M\otimes_\Z\zp\subseteq M'\otimes_\Z\zp\subseteq \Z_K\otimes_\Z\zp
$$
The proof proceeds by comparison of the indices of the two larger modules with respect to 
$\zp[\t]$.

By Lemma \ref{numerators}, The family $\alpha_{i,j,k}$ is a $\zp$-basis of $\zp[\t]$. Hence, just by adding the $v_p$-value of all denominators of the basis of $M'$ we get:
\begin{equation}\label{index1}
v_p\left(\left(M'\otimes_\Z\zp\colon \zp[\t]\right)\right)=\rho\sum_{i=1}^m\left(\lfloor y_{i,1}\rfloor+\cdots+\lfloor y_{i,\ell_i}\rfloor\right)\deg(g_i).
\end{equation}

On the other hand, by Corollary \ref{regreg}, $f$ is $p$-regular with respect to the choices of all $p$-adic irreducible factors $\phi\in\Z_p[x]$ of all polynomials $g_1,\dots,g_m$. These factors $\phi$ act as representatives of the pairwise different irreducible factors of $f$ modulo $p$.

The theorem of the index (Theorem \ref{index}) shows that 
$$
v_p\left(\left(\Z_K\otimes_\Z\zp\colon \zp[\t]\right)\right)=\ind_p(f)=\sum\nolimits_{\phi\mid g_1\cdots g_m}\ind_\phi(f).
$$
Let us separate the latter sum according to the different polynomials $g_i$:
$$%\begin{equation}\label{index2}
\ind_p(f)=\sum_{i=1}^m\sum_{\phi\mid g_i}\ind_\phi(f).
$$%\end{equation}

Now, by Theorem \ref{equalpolygons}, the Newton polygons $N_\phi^-(f)$ for the different factors of the same $g_i$ coincide with $E_\rho(N_{g_i}^-(f))$. By the definition of the $\phi$-index (Definition \ref{phindex}), we have
$$
\ind_\phi(f)=\deg(\phi) \,\left(\lfloor \rho y_{i,1}\rfloor+\cdots+\lfloor \rho y_{i,\ell_i}\rfloor\right).
$$
Hence, we may rewrite the expression for $\ind_p(f)$ as:
\begin{equation}\label{index3}
\ind_p(f)=\sum_{i=1}^m\left(\lfloor \rho y_{i,1}\rfloor+\cdots+\lfloor \rho y_{i,\ell_i}\rfloor\right)\deg g_i.
\end{equation}

Therefore, $M'\otimes_\Z\zp= \Z_K\otimes_\Z\zp$ if and only if the two computations of (\ref{index1}) and  (\ref{index3}) coincide, which is equivalent to
\begin{equation}\label{comparison}
\rho\left(\lfloor y_{i,1}\rfloor+\cdots+\lfloor y_{i,\ell_i}\rfloor\right)=
\lfloor \rho y_{i,1}\rfloor+\cdots+\lfloor \rho y_{i,\ell_i}\rfloor,
\end{equation}
for all $1\le i\le m$.

Any of the conditions (a) or (b) in Theorem \ref{main} implies (\ref{comparison}). In fact, if $N$ is squarefree we have $\rho=1$, and if all slopes of all Newton polygons $N_{g_i}^-(f)$ are integers, then all rational numbers $y_{i,j}$ are integers too. In both cases, (\ref{comparison}) is obvious.

Conversely, suppose $N$ is not squarefree and there exists a slope $-\lambda$ of some $N_{g_i}^-(f)$ which is not an integer; that is, $\lambda=h/e$ with $h,e$ positive coprime integers and $e>1$. 

Clearly, $ \rho\lfloor y_{i,j}\rfloor\le\lfloor \rho y_{i,j}\rfloor$ for all $i,j$; hence, it suffices to show the existence of some $1\le j\le \ell_i$ for which the inequality is strict, to conclude that (\ref{comparison}) does not hold. 

Since $N$ is not squarefree, there exists a prime divisor $p$ of $N$ with $\rho=v_p(N)>1$.
Let $(s,u)$ be the left endpoint of the side of slope $\lambda$ of $N_{g_i}^-(f)$. We have $y_{i,s}=u\in\Z_{>0}$, and 
$$
y_{i,s+k}=u-k\lambda=u-\dfrac{kh}e=\dfrac{ue-kh}e,\quad 1\le k<e.
$$
Since $h$ and $e$ are coprime, there exists $1\le k<e$ such that $kh\equiv 1 \md{e}$. For this value of $k$ we may write the positive numerator of the last fraction as $ue-kh=e-1+eb$, for some non-negative integer $b$. For $j=s+k$ we have
$$
\lfloor \rho y_{i,j}\rfloor=\rho b+\lfloor \rho (e-1)/e\rfloor>\rho b=\rho\lfloor y_{i,j}\rfloor,
$$
because $\rho>1$.
This ends the proof of the theorem.
\hfill{$\Box$}

\section{An example}
Let us illustrate the practical performance of Ore's method modulo $N$ in a concrete example. All computations have been done in a PC using Magma V2.19-7.

Consider the following irreducible polynomial of degree six:
$$
f(x)=(x^2+x+2)^2\left(x^2+x+2+a(a-1)\right)-4a^3,
$$
where $a=pq^2$, and $p,q$ are the prime numbers $p=281474976710677$, $q=1099511627791$.

Once we apply trial division by $2$, $3$ and $5$ (the primes less than or equal to the degree of $f$), we get:
$$
\dsc(f)=2^{8}3^7N,
$$
where $N$ is a $2685$-bit integer. For the primes $p=2,3$ we compute $p$-integral bases with the traditional methods. The remaining task is the computation of an $N$-integral basis. 

The previous factorization of $N$ would require a lot of time. Even the squarefree factorization of $N$ has a sensible cost: it takes 8001.21 seconds.

However, Ore's metod applied to the modulus $N$ is able to compute an $N$-integral basis in a much shorter time.

In fact, while trying to compute the squarefree decomposition of $f$ modulo $N$, the routine SFD0 detects the number $a^2$ as a proper divisor of $N$. From this information, we deduce the following splitting of $N$ into a product of powers of coprime base factors:  
$$
N=a^{12}N_1,
$$
where $N_1$ is a $1149$-bit integer. Hence, we consider $[a,N_1]$ as a list of moduli to which the method must be applied and we start over.

When we try to compute the squarefree decomposition of $f$ modulo $N_1$, the routine SFD0 detects $13$ as a proper divisor of $N_1$. This leads to the following splitting of $N_1$ into a product of powers of coprime base factors:  
$$
N_1=13^3N_2,
$$
where $N_2$ is a $1138$-bit integer. At this moment, we have $[a,13,N_2]$ as a list of moduli to which the method must be applied.

When we try to compute the squarefree decomposition of $f$ modulo $N_2$, the routine SFD0 detects a proper divisor $b$ of $N_2$, leading to:  
$$
N_2=b^2b',
$$
where $b$ is a 376-bit integer and $b'$ is a 387-bit integer coprime to $b$:
\begin{align*}
b=&\,1122564279191696029517040619451061074971851154240399803500\\&\,6152942172293886972367
007941839422932978008036753065979,\\
b'=&\,25863880992576676520072615872152447167521
78991256168048379\\&\,2872992197630558791833856045295665609912136344999241131327.
\end{align*}

We get $[a,13,b,b']$ as the list of coprime moduli $m$ for which we want to compute an $m$-integral basis. For each of these moduli, $f$ is $m$-regular and the method computes a candidate of $m$-integral basis without detecting a further splitting of the modulus. 

For the moduli $13$, $b$ and $b'$ we get Newton polygons with non-integer slopes. 
Nevertheless, for the modulus $a$ the routine SFD0 considers $f\equiv (x^2+x+2)^3\md{a}$ as a squarefree decomposition of $f$ modulo $a$ and for $g=x^2+x+2$, the Newton polygon $N_g^-(f)$ is one-sided of length $3$ and it has integer slope $-1$. 

The first phase is over and it took a total time of 0.11 seconds. Now, by Theorem \ref{main}, we have in our hands a $13$-integral basis and an $a$-integral basis (although $a$ is not squarefree), and we need only to compute the squarefree factorization of $b$ and $b'$ to decide if we have terminated already (if they are both squarefree), or we need to start over with some divisors of $b,b'$ as new moduli. 

The point is that these moduli $b,b'$ are small enough to obtain their squarefree factorization in a reasonable time: 0.21 and 14.91 seconds, respectively. It turns out that $b$ and $b'$ are both squarefree and the computation terminates in a total accumulated time of 15.23 seconds.

\end{document}